\documentclass[12pt]{amsart}
\usepackage{amssymb,mathrsfs,color,bbold}

\theoremstyle{plain}
\newtheorem{theorem}{Theorem}[section]

\newtheorem{corollary}[theorem]{Corollary}

\theoremstyle{definition}
\newtheorem{definition}[theorem]{Definition}

\usepackage{amssymb}
\newcommand{\upcite}[1]{{\rm \textsuperscript{\textsuperscript{\cite{#1}}}}}

\newcommand{\term}[1]{{\textit{\textbf{#1}}}}

\newcommand{\abs}[1]{\lvert#1\rvert}
\newcommand{\norm}[1]{\lVert#1\rVert}
\newcommand{\bigabs}[1]{\bigl\lvert#1\bigr\rvert}
\newcommand{\bignorm}[1]{\bigl\lVert#1\bigr\rVert}

\title[uaw-Dunford-Pettis operator]{Some results on unbounded absolute weak Dunford-Pettis operators}

\date{\today}


\author[H. Li]{Hui Li}
\address{School of Mathematics, Southwest Jiaotong University,
Chengdu, Sichuan,
China, 610000.}
\email{lihuiqc@my.swjtu.edu.cn}

\author[Z. Chen]{Zili Chen}
\address{School of Mathematics, Southwest Jiaotong University, Chengdu, Sichuan,
China, 610000.}
\email{zlchen@home.swjtu.edu.cn}

\begin{document}
\begin{abstract}
In this paper, we characterize  Banach lattices on which each Dunford-Pettis operator (or weak Dunford-Pettis) is unbounded absolute weak Dunford-Pettis operator and the converse. 
\end{abstract}

\maketitle

\section{Introduction}

The notion of unbounded order convergence (uo-convergence, for short) was firstly introduced by Nakano in \cite{N:48}, then it was used and systematically investigated in \cite{G:14,GTX:16,GX:14, LC:18,W:77}.  After that, Bahramnezhad and Azar proposed the definition of unbounded order continuous operators in \cite{BA:18}. A closely related notion of unbounded norm convergence (un-convergence, for short) was introduced and systematically studied in \cite{DOT:16,KMT:16,TR:04}. In \cite[Section 9]{KMT:16}, Kandi\'c et al. gave the definition of (sequentially) un-compact operators and obtained the relationships between weakly compact operators and sequentially un-compact operators. Recently, Zabeti  in \cite{Zabeti:17} proposed a new so-called unbounded version convergence (uaw-convergence). And, uaw-Dunford-Pettis operators were introduced and investigated in \cite{EGZ:17}. 

In this paper, we will establish some results on uaw-Dunford-Pettis operators. We first present some necessary and sufficient conditions for positive Dunford-Pettis operators being uaw-Dunford-Pettis. More precisely, we will prove that each positive Dunford-Pettis operator from a Banach lattice $E$ into arbitrary Banach lattice $F$ is uaw-Dunford-Pettis if and only if the norm of $E'$ is order continuous or $F=\{0\}$ (Theorem \ref{DP_uawDP}). We will also give a characterization of Banach  lattice $E$ on which each positive operator $T: E\rightarrow \ell_{1}$ is uaw-Dunford-Pettis (Theorem \ref{po-uawDP}). 
After that, we will investigate Banach lattices under which each uaw-Dunford-Pettis operator is Dunford-Pettis.  
Finally, we will present the relationships between weak Dunford-Pettis operators and uaw-Dunford-Pettis operators.  Whenever Banach lattice $E$ is Dedekind $\sigma$-complete, we will establish that $E$ is reflexive if and only if each positive weak Dunford-Pettis operator from $E$ into $E$ is an uaw-Dunford-Pettis operator (Theorem \ref{wdp-uawdp}). We will also give some sufficient conditions under which each positive uaw-Dunford-Pettis operator is weak Dunford-Pettis (Theorem \ref{uawdp-wdp1}).

\section{Preliminaries}
To state our results, we need to recall some definitions. Recall that a Riesz space $E$ is an ordered vector space in which $\sup(x,y)$ exists for every $x, y\in E$.  A sequence $(u_{n})$ of a Riesz space is called \term{disjoint} whenever $n\ne m$ implies $u_{n}\perp u_{m}$.  A  Banach lattice is a Banach space $(E, \norm{\cdot})$ such that $E$ is a Riesz lattice and its norm satisfies the following property: for each $x, y\in E$ with $\abs{x}\le \abs{y}$, we have $\norm{x}\le \norm{y}$.
 By Theorem 4.1 of \cite{AB:06},  if $E$ is a Banach lattice, then its norm dual $E'$ is also a Banach lattice.
 
A norm $\norm{\cdot}$ of a  Banach lattice $E$ is \term{order continuous} if for each net $(x_{\alpha})$ in $E$ with $x_{\alpha}\downarrow 0$, one has $\norm{x_{\alpha}}\downarrow 0$.
 A Banach lattice $E$ is said to be a \term{KB-space} whenever every increasing norm bounded sequence of $E_{+}$ is norm convergent. Every KB-space has an order continuous norm.  A  Banach space is said to have the \term{Schur property} whenever  every weak convergent sequence is norm convergent, i.e., whenever $x_{n}\xrightarrow{w}0$ implies $\norm{x_{n}}\rightarrow 0$.

Recall that an operator $T$ from a Banach space $X$ to a Banach space $Y$ is \term{Dunford-Pettis} if  it maps weakly null sequences of $X$ to norm null sequences of $Y$, and is \term{weak Dunford-Pettis} if $f_{n}(T(x_{n}))\rightarrow 0$ for any weakly null sequence $(x_{n})$ in $X$ and any weakly null sequence $(f_{n})$ in $Y'$.

 Recall that a net $(x_{\alpha})$ in a Banach lattice $E$ is said to be \term{unbounded absolutely weakly convergence} to $x\in E$, written as $x_{\alpha}\xrightarrow{uaw}x$, if for any $u\in E_{+}$,$\abs{ x_{\alpha}-x}\wedge u\xrightarrow{w} 0$ holds.  
 
\begin{definition} \upcite{EGZ:17}
 An operator $T$ from a Banach lattice $E$ into a Banach space $X$ is said to be an \term{unbounded absolute weak Dunford-Pettis} (uaw-Dunford-Pettis, for short) if for every norm bounded sequence $(x_{n})$ in $E$, $x_{n}\xrightarrow{uaw} 0$ implies $\bignorm{Tx_{n}}\rightarrow 0$. 
\end{definition} 
 
 Every uaw-Dunford-Pettis operator is continuous. In fact, if $T:E\rightarrow X$ is an uaw-Dunford-Pettis operator and $\norm{x_{n}}\rightarrow  0$, then for each $u\in E_{+}$, $\norm{\abs{x_{n}}\wedge u}\le \norm{\abs{x_{n}}}=\norm{x_{n}}$, i.e., $\norm{\abs{x_{n}}\wedge u}\rightarrow 0$, then $\abs{x_{n}}\wedge u\xrightarrow{w} 0$. That is,  $x_{n}\xrightarrow {uaw} 0$, and so $\bignorm{Tx_{n}}\rightarrow 0$.

All operators in this paper are assumed to be continuous. We refer to ~\cite{AB:06,MN:91} for all unexplained terminology and standard facts on vector and Banach lattices. All vector lattices in this paper are
assumed to be Archimedean.

\section{The relationships with Dunford-Pettis operators}

 There exist operators which are Dunford-Pettis but not uaw-Dunford-Pettis. For example,  the identity operator $Id_{\ell_{1}}:\ell_{1}\rightarrow \ell_{1}$ is Dunford-Pettis since $\ell_{1}$ has the Schur property, but it is not an uaw-Dunford-Pettis operator. In fact, for the standard basis $(e_{n})$ of $\ell_{1}$, $(e_{n})$ is disjoint, so by Lemma 5 of \cite{Zabeti:17},  $e_{n}\xrightarrow{uaw} 0$. However, $\bignorm{Id_{\ell_{1}}(e_{n})}=\norm{e_{n}}=1$. 
  
 The following theorem gives a characterization of Banach lattices $E$ and $F$ under which each positive Dunford-Pettis operator $T:E\rightarrow F$ is uaw-Dunford-Pettis.
  
\begin{theorem}\label{DP_uawDP}
Let $E$ and $F$ be Banach lattices. Then the following assertions are equivalent:
\begin{itemize}
\item[(1)] Each positive Dunford-Pettis operator $T:E\rightarrow F$ is uaw-Dunford-Pettis.
\item[(2)] Each positive compact operator $T:E\rightarrow F$ is  uaw-Dunford-Pettis.
\item[(3)] One of the following conditions is valid:
\begin{itemize}
\item[(i)] The norm of $E'$ is order continuous.
\item[(ii)] $F=\{0\}$.
\end{itemize}
\end{itemize}
\end{theorem}

\begin{proof}
$(1)\Rightarrow (2)$ It is obvious, since each compact operator is Dunford-Pettis.

$(2)\Rightarrow (3)$ Assume by way of contradiction that the norm of $E'$ is not order continuous and $F\ne \{0\}$. We have to construct a compact operator which is not uaw-Dunford-Pettis.

Since the norm of $E'$ is not order continuous,  it follows from Theorem 2.4.14 and Proposition 2.3.11 of \cite{MN:91} that $\ell_{1}$ is a closed sublattice of $E$ and there exists a positive projection $P:E\rightarrow \ell_{1}$. On the other hand, since $F\ne \{0\}$, there exists a vector $0<y \in F_{+}$. Define the operator $S:\ell_{1}\rightarrow F$ as follows:
$$S(\lambda_{n})=(\sum_{n=1}^{\infty} \lambda_{n})y$$
for each $(\lambda_{n})\in \ell_{1}$. Obviously,  the operator $S$ is  well defined. Let 
$$T=S\circ P: E\rightarrow \ell_{1} \rightarrow F,$$
then $T$ is a compact operator since $S$ is a finite rank operator ( rank is 1). But $T$ is not an uaw-Dunford-Pettis operator. Let $(e_{n})$ be the canonical basis of $\ell_{1}$. Obviously, $(e_{n})$ is disjoint, by Lemma 5 of \cite{Zabeti:17}, we know that $e_{n}\xrightarrow{uaw} 0$. However, $\bignorm{T(e_{n})}=\norm{y}>0$. Hence, $T$ is not an uaw-Dunford-Pettis operator.

$(3)(i)\Rightarrow (1)$ Follows from Proposition 1 of \cite{EGZ:17}. 

$(3)(ii)\Rightarrow (1)$ Obvious.
\end{proof}

Whenever $E=F$ in the Theorem \ref{DP_uawDP}, we get the following characterization:

\begin{corollary}
Let $E$ be a Banach lattice. Then the following assertions are equivalent:
\begin{itemize}
\item[(1)] Each positive Dunford-Pettis operator $T:E\rightarrow E$ is uaw-Dunford-Pettis.
\item[(2)] Each positive compact operator $T:E\rightarrow E$ is  uaw-Dunford-Pettis.
\item[(3)] The norm of $E'$ is order continuous.
\end{itemize}
\end{corollary}

 The following theorem gives a characterization of Banach lattice $E$ for which each positive operator $T:E\rightarrow \ell_{1}$ is uaw-Dunford-Pettis. 
 
\begin{theorem}\label{po-uawDP}
Let $E$ be a Banach lattice, then the following assertions are equivalent:
\begin{itemize}
\item[(1)] Each positive operator from $E$ into $\ell_{1}$ is uaw-Dunford-Pettis.
\item[(2)] The norm of $E'$ is order continuous.
\end{itemize}
\end{theorem}

\begin{proof}
$(1)\Rightarrow (2)$ Assume by way of contradiction that the norm of $E'$ is not order continuous. 
 Then it follows from Theorem 116.1 of \cite{Zaanen:83} that there exists a norm bounded disjoint sequence $(u_{n})$ of positive elements in $E$ which does not  weakly convergence to zero.  Without loss of generality, we may assume that $\norm{u_{n}}\le 1$ for any $n$. And there exist $\varepsilon >0$ and $0\le \phi\in E'$ such that $\phi(u_{n})>\varepsilon$ for all $n$.  Then by Theorem 116.3 of \cite{Zaanen:83}, we know that the components $\phi_{n}$ of $\phi$ in the carriers $C_{u_{n}}$ form an order bounded disjoint sequence in $(E')_{+}$ such that 
 
\centerline{$\phi_{n}(u_{n})=\phi(u_{n})$\ for\ all\ $n$\quad and \quad $\phi_{n}(u_{m})=0$\ if \ $n\ne m$.} 

Define the positive operator $T:E\rightarrow \ell_{1}$ as follows:
$$T(x)=\left(\frac{\phi_{n}(x)}{\phi(u_{n})}\right)_{n=1}^{\infty}$$
for all $x\in E$. Since 

$$\sum_{n=1}^{\infty}\bigabs{\frac{\phi_{n}(x)}{\phi(u_{n})}}\le \frac{1}{\varepsilon}\sum_{n=1}^{\infty}\phi_{n}(\abs{x})\le  \frac{1}{\varepsilon}\phi(\abs{x})$$
holds for all $x\in E$, the operator $T$ is well defined and it is also easy to see that $T$ is a positive operator. Hence $T$ is an uaw-Dunford-Pettis operator. For the norm bounded disjoint sequence $(u_{n})$, by Lemma 5 of \cite{Zabeti:17}, we know that $u_{n}\xrightarrow{uaw} 0$. 
However, let $(e_{n})$ be the standard basis of $\ell_{1}$, then $\bignorm{T(u_{n})}=\norm{e_{n}}=1$, which is a contradiction.  Therefore, the norm of $E'$ is order continuous.

$(2)\Rightarrow (1)$ Since $\ell_{1}$ has the Schur property,  each positive operator $T$ from $E$ into $\ell_{1}$ is Dunford-Pettis. And since the norm of $E'$ is order continuous,  by Proposition 1 of \cite{EGZ:17}, we obtain that $T$ is uaw-Dunford-Pettis.
\end{proof}

Based on Theorem 5.29 of \cite{AB:06} and Theorem 2.9 of \cite{FKM:17}, we get the following conclusion.

\begin{corollary}
Let $E$ be a Banach lattice, then the following assertions are equivalent:
\begin{itemize}
\item[(1)] The norm of $E'$ is order continuous.
\item[(2)] Each positive operator from $E$ into $\ell_{1}$ is uaw-Dunford-Pettis.
\item[(3)] Each positive operator from $E$ into $\ell_{1}$ is weakly compact, and hence compact.
\item[(4)] Each positive operator from $E$ into $\ell_{1}$ is semi-compact.
\end{itemize}
\end{corollary}

A  Banach lattice is said to have \term{weakly sequentially continuous lattice operations} whenever $x_{n}\xrightarrow{w} 0$ implies $\abs{x_{n}}\xrightarrow{w} 0$. Every $AM$-space has this property.

 The following theorem gives a characterization of Banach lattices $E$ and $F$ for which each uaw-Dunford-Pettis operator $T:E\rightarrow F$ is Dunford-Pettis. 

\begin{theorem}\label{uawdp-dp1}
Let $E$ and $F$ be  Banach lattices. Each uaw-Dunford-Pettis operator $T:E\rightarrow F$ is Dunford-Pettis if one of the following assertions is valid:
\begin{itemize}
\item[(1)] The lattice operations in $E$ are weakly sequentially continuous.
\item[(2)] $E$ is discrete with an order continuous norm.
\item[(3)] $T$ is positive and $F$ is discrete with an order continuous norm.
\end{itemize}
\end{theorem}
\begin{proof}
$(1)$ Let $(x_{n})$ be a weakly null sequence in $E$. Since the lattice operations in $E$ are weakly sequentially continuous,  we have $\abs{x_{n}}\xrightarrow{w} 0$. Then for each $u\in E_{+}$, $\abs{x_{n}}\wedge u\xrightarrow{w} 0$, i.e., $x_{n}\xrightarrow{uaw} 0$. Since $T$ is an uaw-Dunford-Pettis operator, we get  $\bignorm{T(x_{n})}\rightarrow 0$. Hence, the operator $T$ is Dunford-Pettis.

$(2)$ Suppose that $E$ is discrete with an order continuous norm, then by Corollary 2.3 of \cite{CW:98}, the lattice operations in $E$ are weakly sequentially continuous. Hence,  following  from (1), we get the result.

$(3)$ Let $T:E\rightarrow F$ be a positive uaw-Dunford-Pettis operator and $W$ be a relatively weakly compact set in $E$, we have to show $T(W)$ is a relatively  compact set in $F$. Let $A$ be the solid hull of $W$ in $E$. For every disjoint sequence $(x_{n})$ in $A$, by Lemma 5 of \cite{Zabeti:17}, we know that $x_{n}\xrightarrow{uaw} 0$. Since $T$ is uaw-Dunford-Pettis, we get that $\bignorm{T(x_{n})}\rightarrow 0$. Then by Theorem 4.36 of  \cite{AB:06}, for each $\varepsilon>0$, there exists some $u\in E_{+}$ lying in the ideal generated by $A$ such that $\bignorm{T[(\abs{x}-u)^{+}]}<\varepsilon$ holds for all $x\in A$. Following from the equality $\abs{x}=\abs{x}\wedge u+(\abs{x}-u)^{+}$, we have 
$$T(\abs{x})=T(\abs{x}\wedge u)+T[(\abs{x}-u)^{+}].$$
Let $V$ be the closed unit ball of $F$. Then 
$$T(\abs{x})\in [-T(u), T(u)] + \varepsilon \cdot V$$
for all $x\in A$. Since $T$ is a positive operator, $\abs{T(x)}\le T(\abs{x})$. It is easy to see that the set $[-T(u), T(u)] + \varepsilon \cdot V$ is a solid set in $F$. Hence, 
$$T(x)\in [-T(u), T(u)] + \varepsilon \cdot V$$
 for all $x\in A$, and then
$$T(W)\subset [-T(u), T(u)] + \varepsilon \cdot V.$$
Since $F$ is discrete with an order continuous norm, $[-T(u), T(u)]$ is norm compact. Hence, $T(W)$ is a relatively compact set in $F$. Thus $T$ is a Dunford-Pettis operator.
\end{proof}

\begin{corollary}\label{uawdp-dp2}
Let $E$ and $F$ be Banach lattices such that the norm of $E'$ is order continuous and $F$ is discrete or its lattice operations are weakly sequentially continuous. Then the following assertions are equivalent:
\begin{itemize}
\item[(1)] Each positive uaw-Dunford-Pettis operator $T:E\rightarrow F$ is Dunford-Pettis.
\item[(2)] One of the following assertions is valid:
\begin{itemize}
\item[(i)] The lattice operations in $E$ are weakly sequentially continuous.
\item[(ii)] The norm of $F$ is order continuous.
\end{itemize}
\end{itemize}
\end{corollary}
\begin{proof}
$(2)(i)\Rightarrow (1)$ Follows from Theorem \ref{uawdp-dp1}(1).

$(2)(ii)\Rightarrow (1)$ Based on Corollary 2.3 of \cite{CW:98}, if $F$ has an order continuous norm and the lattice operations of it are weakly sequentially continuous, then $F$ is also discrete. Therefore, following  from Theorem \ref{uawdp-dp1}(3), we get the result.

$(1)\Rightarrow (2)$ Let $S: E\rightarrow F$ be a operator which satisfies $0\le S \le T$ and $T: E\rightarrow F$ is a Dunford-Pettis operator. Since the norm of $E'$ is order continuous, by Theorem \ref{DP_uawDP}, we get that the operator $T$ is uaw-Dunofrd-Pettis. Now we claim that $S$ is also uaw-Dunofrd-Pettis, i.e., uaw-Dunford-Pettis opertors satisfy domination. In fact, if  $x_{n}\xrightarrow{uaw} 0$ holds in $E$, then it is easy to see that $\abs{x_{n}}\xrightarrow{uaw} 0$. And so $\bignorm{T(\abs{x_{n}})}\rightarrow 0$ holds in $F$. By using the inequalities $\abs{S(x_{n})}\le S(\abs{x_{n}})\le T(\abs{x_{n}})$, we get that $\bignorm{S(x_{n})}\le \bignorm{T(\abs{x_{n}})}$ for all $n$. That is, $S$ is an uaw-Dunford-Pettis operator. Then $S$ is a Dunford-Pettis operator. Following from Theorem 2 of \cite{W:96}, the lattice operations in $E$ are weakly sequentially continuous or the norm of $F$ is order continuous.
\end{proof}

\begin{corollary}\label{AME}
Let $E$ be an $AM$-space. Then every operator $T$ from $E$ into arbitrary Banach space is uaw-Dunford-Pettis  if and only if $T$ is  Dunford-Pettis.
\end{corollary}

\begin{proof}
Follows from Remark 1 of \cite{EGZ:17}.
\end{proof}

\section{The relationships with weak Dunford-Pettis operators }

Recall that a  Banach space $X$ is said to have the \term{Dunford-Pettis property} whenever $x_{n}\xrightarrow{w} 0$ in $X$ and $x'_{n}\xrightarrow{w} 0$ in $X'$ imply $x'_{n}(x_{n})\rightarrow 0$. $AL$-space and $AM$-space have the Dunford-Pettis property (\cite[Theorem 5.85]{AB:06}). Obviously, if $X$ has the Dunford-Pettis property, then every continuous operator from $X$ to a Banach space $Y$ is weak Dunford-Pettis.

Since each Dunford-Pettis operator is weak Dunford-Pettis, the identity operator $Id_{\ell_{1}}:\ell_{1}\rightarrow \ell_{1}$ is also the example which is weak Dunford-Pettis but not uaw-Dunford-Pettis. Next, we give a characterization of reflexive Banach lattice for which each positive weak Dunford-Pettis operator from $E$ into $E$ is uaw-Dunford-Pettis operator.

\begin{theorem}\label{wdp-uawdp}
Let $E$ be a Dedekind $\sigma$-complete Banach lattice. Then the following assertions are equivalent:
\begin{itemize}
\item[(1)]$E$ is reflexive. 
\item[(2)] Each positive weak Dunford-Pettis operator from $E$ into $E$ is uaw-Dunford-Pettis.
\end{itemize}
\end{theorem}
\begin{proof}
$(1)\Rightarrow (2)$ Since $E$ is reflexive, each weak Dunford-Pettis operator $T$ from $E$ into $E$ is Dunford-Pettis. Based on Theorem 4.70 of \cite{AB:06},  the norm of $E'$ is order continuous. Then by Theorem \ref{DP_uawDP}, we know $T$ is uaw-Dunford-Pettis.

$(2)\Rightarrow (1)$ We first claim that  the norm of $E$ is order continuous. Otherwise,  
it follows from  Corollary 2.4.3 of \cite{MN:91}  that $E$ contains a sublattice which is isomorphic to $\ell_{\infty}$ and there exists a positive projection $P:E\rightarrow \ell_{\infty}$. Let $S: \ell_{\infty}\rightarrow E$ be the canonical injection of $\ell_{\infty}$ into $E$. Define the operator $T$ as follows:

$$T=S\circ P: E\rightarrow \ell_{\infty} \rightarrow E.$$

Since $\ell_{\infty}$ has the Dunford-Pettis property,  $T$ is weak Dunford-Pettis operator. Hence, $T$ is uaw-Dunford-Pettis. Let $(e_{n})$ be the standard basis of $\ell_{\infty}$. Similarily to the proof of Theorem \ref{DP_uawDP}, $e_{n}\xrightarrow{uaw} 0$.  However, $\bignorm{T(e_{n})}=\norm{e_{n}}=1>0$, which is a contradiction. Therefore, $E$ has an order continuous norm.

Next, we prove $E$ is a KB-space. If not, it follows from  Theorem 2.4.12 of  \cite{MN:91} that $E$ contains a sublattice which is isomorphic to $c_{0}$ and there exists a positive projection $P:E\rightarrow c_{0}$. Let $S: c_{0}\rightarrow E$ be the canonical injection of $c_{0}$ into $E$. Define the operator $T$ as follows:

$$T=S\circ P: E\rightarrow c_{0}\rightarrow E.$$

Since $c_{0}$ has the Dunford-Pettis property,  $T$ is a weak Dunford-Pettis operator.  Let $(e_{n})$ be  the standard basis of $c_{0}$. Similarly, $e_{n}\xrightarrow{uaw} 0$.  However, $\bignorm{T(e_{n})}=\norm{e_{n}}=1>0$, we get that $T$ is not an uaw-Dunford-Pettis operator, which is a contradiction. Hence, $E$ is KB-space.

At last, we show that the norm of $E'$ is order continuous. If not, it follows from Theorem 2.4.14 and Proposition 2.3.11 of \cite{MN:91} that $E$ contains a sublattice  which is isomorphic to $\ell_{1}$  and there exists a positive projection $P:E\rightarrow \ell_{1}$. Define the operator $T$ as follows:

$$T=S\circ P: E\rightarrow \ell_{1}\rightarrow E.$$

Since $\ell_{1}$ has the Dunford-Pettis property,  $T$ is a weak Dunford-Pettis operator.  Let $(e_{n})$ be the standard basis of $\ell_{1}$. Similarly, $e_{n}\xrightarrow{uaw} 0$.  However, $\bignorm{T(e_{n})}=\norm{e_{n}}=1>0$, we obtain $T$ is not an uaw-Dunford-Pettis operator, which is a contradiction. Hence, $E'$ has an order continuous norm.

Following from Theorem 4.70 of \cite{AB:06}, we obtain that $E$ is reflexive. 
\end{proof}

Whenever $E\ne F$ in Theorem \ref{wdp-uawdp}, we get the following conclusions.
\begin{corollary}
Let $E$ and $F$ be Banach lattices. If the norm of $E'$ is order continuous and $F$ is reflexive, then each  weak Dunford-Pettis operator from $E$ into $F$ is uaw-Dunford-Pettis operator.
\end{corollary}
\begin{proof}
Similarly to the proof of $(1)\Rightarrow (2)$ of the Theorem \ref{wdp-uawdp}. Since $F$ is reflexive, each  weak Dunford-Pettis operator $T$ from $E$ into $F$ is Dunford-Pettis. By  Theorem \ref{DP_uawDP}, we get that $T$ is uaw-Dunford-Pettis.
\end{proof}

\begin{theorem}\label{wdp-uawdp2}
Let $E$  and $F$ be Banach lattices. If each weak Dunford-Pettis operator is uaw-Dunford-Pettis operator, then one of the following assertion is valid:
\begin{itemize}
\item[(1)] The norm of $E'$ is order continuous.
\item[(2)] The norm of $F$ is order continuous.
\end{itemize}
\end{theorem}
\begin{proof}
It suffices to establish that if the norm of $E'$ is not order continuous, then $F$ has an order continuous norm.

Since the norm of $E'$ is not order continuous,  it follows from Theorem 2.4.14 and Proposition 2.3.11 of \cite{MN:91} that $\ell_{1}$ is a closed sublattice of $E$ and there exists a positive projection $P:E\rightarrow \ell_{1}$. We need to show that $F$ has an order continuous norm. By Theorem 4.14 of \cite{AB:06}, it suffices to show that each order bounded disjoint sequence $(y_{n})$ is norm convergent to 0  in $F$. 

Define the operator $S:\ell_{1}\rightarrow F$ as follows:
$$S(\lambda_{n})=\sum_{n=1}^{\infty}\lambda_{n}y_{n}$$
for each $(\lambda_{n})\in \ell_{1}$. Obviously, it is  well defined. Let 
$$T=S\circ P: E\rightarrow \ell_{1} \rightarrow F.$$
Since $\ell_{1}$ has the Dunford-Pettis property,  $T$ is a weak Dunford-Pettis operator. Then $T$ is uaw-Dunford-Pettis. Let $(e_{n})$ be the standard basis of $\ell_{1}$,   $e_{n}\xrightarrow{uaw} 0$, so, $\bignorm{T(e_{n})}=\norm{y_{n}}\rightarrow 0$. Hence, $F$ has an order continuous norm.
\end{proof}

At last, we give a characterization of  Banach lattices for which each positive uaw-Dunford-Pettis operator from $E$ into $F$ is weak Dunford-Pettis operator. 

Recall that a Banach lattice is said to have \term{AM-compactness property} if every weakly compact operator from $E$ to an arbitrary Banach space is AM-compact.  The Banach lattices $c_{0}$, $\ell_{1}$, $c$, and $c'$ have AM-compactness property. We have the following conclusion.

\begin{theorem}\label{uawdp-wdp1}
Let $E$ and $F$ be  Banach lattices . Each positive uaw-Dunford-Pettis operator $T:E\rightarrow F$ is weak Dunford-Pettis if one of the following assertions is valid:
\begin{itemize}
\item[(1)] The lattice operations in $E$ are weakly sequentially continuous.
\item[(2)] $F$ is discrete with an order continuous norm.
\item[(3)] $F$ has AM-compact property.
\end{itemize}
\end{theorem}
\begin{proof}
Since each Dunford-Pettis operator is weak Dunford-Pettis, it follows from Theorem \ref{uawdp-dp1}, if the lattice operations in $E$ are weakly sequentially continuous or $F$ is discrete with an order continuous norm, every positive uaw-Dunford-Pettis operator $T:E\rightarrow F$ is weak Dunford-Pettis.

Next, we only need to show if $F$ has AM-compact property, the assertion is valid. Let $T:E\rightarrow F$ be a positive uaw-Dunford-Pettis operator and $W$ be a relatively weakly compact set in $E$, we have to show $T(W)$ is a Dunford-Pettis set in $F$. Let $A$ be the solid hull of $W$ in $E$ and $V$ be the closed unit ball of $F$. It follows from the proof of Theorem \ref{uawdp-dp1}(3),  for each $\varepsilon>0$, there exists some $u\in E_{+}$ lying in the ideal generated by $A$ such that  
$$T(W)\subset [-T(u), T(u)] + \varepsilon \cdot V.$$
Since $F$ has AM-compact property, based on  Proposition 3.1 and Lemma 4.1 of \cite{AB:13}, we get that  $T(W)$ is a Dunford-Pettis set in $F$. Therefore, following from Theorem 5.99 of \cite{AB:06}, $T$ is a weak Dunford-Pettis operator.
\end{proof}

\end{document}